\documentclass[pdflatex,sn-mathphys-num]{sn-jnl}% Math and Physical Sciences Numbered Reference Style
\usepackage{graphicx}%
\usepackage{multirow}%
\usepackage{amsmath,amssymb,amsfonts}%
\usepackage{amsthm}%
\usepackage{mathrsfs}%
\usepackage[title]{appendix}%
\usepackage{xcolor}%
\usepackage{textcomp}%
\usepackage{manyfoot}%
\usepackage{booktabs}%
\usepackage{algorithm}%
\usepackage{algorithmicx}%
\usepackage{algpseudocode}%
\usepackage{listings}%
\newcommand{\vertiii}[1]{{\left\vert\kern-0.25ex\left\vert\kern-0.25ex\left\vert #1 
        \right\vert\kern-0.25ex\right\vert\kern-0.25ex\right\vert}}

\newcommand{\vertii}[1]{{\left\vert\kern-0.25ex\left\vert #1 
        \right\vert\kern-0.25ex\right\vert}}

\DeclareMathOperator{\rank}{rank}

\DeclareMathOperator{\supp}{supp}

\usepackage{mathtools}

\DeclarePairedDelimiter{\scalprod}{\langle}{\rangle}
\newcommand{\seb}[1]{{\color{red}#1}}

\newcommand{\suchthat}{\ifnum\currentgrouptype=16 \mathrel{}\middle|\mathrel{}\else\mid\fi}

\theoremstyle{thmstyleone}%
\newtheorem{theorem}{Theorem}%  meant for continuous numbers
\newtheorem{proposition}[theorem]{Proposition}% 

\theoremstyle{thmstyletwo}%
\newtheorem{remark}{Remark}%

\theoremstyle{thmstylethree}%
\newtheorem{definition}{Definition}%

\newtheorem{lemma}[theorem]{Lemma}
\newtheorem{corollary}[theorem]{Corollary}
\begin{document}

\title[$L^q$ approximate controllability frequency criterion for linear difference delay equations with distributed delays]{$L^q$ approximate controllability frequency criterion for linear difference delay equations with distributed delays}

%%=============================================================%%
%% GivenName	-> \fnm{Joergen W.}
%% Particle	-> \spfx{van der} -> surname prefix
%% FamilyName	-> \sur{Ploeg}
%% Suffix	-> \sfx{IV}
%% \author*[1,2]{\fnm{Joergen W.} \spfx{van der} \sur{Ploeg} 
%%  \sfx{IV}}\email{iauthor@gmail.com}
%%=============================================================%%

\author*[1]{\fnm{Sébastien } \sur{Fueyo}}\email{sebastien.fueyo@gipsa-lab.fr}

% \author[2,3]{\fnm{Second} \sur{Author}}\email{iiauthor@gmail.com}
% \equalcont{These authors contributed equally to this work.}

% \author[1,2]{\fnm{Third} \sur{Author}}\email{iiiauthor@gmail.com}
% \equalcont{These authors contributed equally to this work.}

\affil*[1]{\orgdiv{Univ. Grenoble Alpes, Inria, CNRS, Grenoble INP, GIPSA-Lab}, \orgaddress{\city{Grenoble}, \postcode{38000}, \country{France}}}

% \affil[2]{\orgdiv{Department}, \orgname{Organization}, \orgaddress{\street{Street}, \city{City}, \postcode{10587}, \state{State}, \country{Country}}}

% \affil[3]{\orgdiv{Department}, \orgname{Organization}, \orgaddress{\street{Street}, \city{City}, \postcode{610101}, \state{State}, \country{Country}}}

%%==================================%%
%% Sample for unstructured abstract %%
%%==================================%%

\abstract{Based on an algebraic point of view and the realization theory developed by Y. Yamamoto, the present paper states a necessary and sufficient criterion, given in the frequency domain, for the $L^q$ approximate controllability in finite time of linear difference delay equations with distributed delays. Furthermore, an upper bound for the minimal time of the $L^q$ approximate controllability is obtained.}


%%\pacs[JEL Classification]{D8, H51}

%%\pacs[MSC Classification]{35A01, 65L10, 65L12, 65L20, 65L70}

\maketitle

\section{Introduction}

The present paper deals with the approximate controllability of linear difference 
delay equations with distributed delays of the form
\begin{equation}
\label{system_lin_formel2}
 x(t)=\sum_{j=1}^NA_jx(t-\Lambda_j)+\int_0^{\Lambda_N} g(s) x(t-s)ds+Bu(t) , \qquad t \ge 0,
\end{equation}
where, given three positive integers $d$, $m$ and $N$, $g(\cdot)$ belongs to $L^{\infty}([0,\Lambda_N],\mathbb{R}^{d \times d}$), $A_1,\dotsc,A_N$ are fixed $d\times d$ 
matrices with real entries, the state $x$ and the control $u$ belong 
to $\mathbb{R}^d$ and $\mathbb{R}^m$ respectively, and $B$ is a fixed $d \times m$ matrix 
with real entries. Without loss of generality, the delays $\Lambda_1, \dotsc, \Lambda_N$ 
are positive real numbers so that $\Lambda_1< \dotsb <\Lambda_N$. 

Delay systems are instrumental to study the properties of some one dimensional hyperbolic partial differential equations (PDEs) \textit{via} the characteristic or the backstepping methods, see for instance \cite{bastin2016stability,CoNg,baratchart,chitour:hal-04228797,auriol2019explicit}. In particular, the paper \cite{auriol2019explicit} goes through the study of System~\eqref{system_lin_formel2} to stabilize such PDEs systems. The stability properties of System~\eqref{system_lin_formel2} are well-known, thanks to the application of the Laplace transform and the theory of almost periodic functions. A necessary and sufficient frequency stability criterion, which is hugely used in the scientific literature, is given in the book \cite{Hale}. Unfortunately, the controllability properties have been less studied. It seems that the only controllability conditions have been obtained for linear difference delay systems without distributed delays (equivalently $g \equiv 0$), see \cite{loiseau2000algebraic,Mazanti2017Relative,Chitour2020Approximate,chitour:hal-03827918,fueyo-chitour}. The addition of distributed delays makes the analysis more difficult and the methodology has to be adapted to give controllability results for linear difference delay equations with distributed delays.

Since System~\eqref{system_lin_formel2} is infinite dimensional, there are different controllability notions depending in particular on the choice of the state space of System~\eqref{system_lin_formel2}. In this paper, we are interested to control state trajectories belonging to the $L^q$ spaces, \textit{i.e.} the spaces of $q$ integrable functions for $q\in [1,+\infty)$. We focus our attention on the $L^q$ approximate controllability in finite time $T>0$, meaning that we can steer the state trajectories of System~\eqref{system_lin_formel2} toward  all targets in $L^q$, as near as we want, when applying a control during a time $T>0$. The method used amounts to study the controllability problem from an algebraic point of view, \textit{i.e.} with the work in some convolution algebras. This idea has been mainly developed in the papers \cite{yamamoto1989reachability,kamen1976module,rouchaleau1972algebraic} and we adopt this framework to study the controllability properties of System~\eqref{system_lin_formel2}. It is worth noting that the algebraic point of view is also often used for identifiability problems, see for instance the paper \cite{belkoura2005identifiabilty}.

The general suitable structure to study the controllability of delay systems is the one given in the two papers of Y. Yamamoto \cite{YamamotoRealization,yamamoto1989reachability}. In particular, the paper \cite{yamamoto1989reachability} states controllability criteria for general systems called \textit{pseudo-rational}, a notion introduced first by Y. Yamamoto. On the one hand, approximate controllability conditions obtained in \cite{yamamoto1989reachability} are given for pseudo-rational systems but the paper does not address the issue of the controllability in finite time. On the other hand, we can find in the introduction of \cite{yamamoto1989reachability} that delay systems with distributed delays are pseudo-rational but there is no proof in the main corpus of this fact. Thus the results of the present paper are threefold. We prove first that System~\eqref{system_lin_formel2} is pseudo-rational in Theorem~\ref{theorem_pseudo_rational}. Secondly, we are able to provide an upper bound for the minimal time of controllability in Theorem~\ref{theorem_temps_minimal}. As a byproduct, applying the results of the paper \cite{yamamoto1989reachability}, we state a necessary and sufficient criterion, obtained in the frequency domain, for the $L^q$ approximate controllability in finite time $T>2d \Lambda_N$ of System~\eqref{system_lin_formel2}, see Theorem~\ref{th:main_result_dist}.

The remaining of the paper is organised as follows. Section~\ref{sec:not} introduces the notations and the distributional algebras needed, while Section~\ref{sec:pre} recalls some basic properties about the well-posedness and the definition of the $L^q$ approximate controllability in finite time $T>0$ for System~\eqref{system_lin_formel2}. Section~\ref{sec:real_theory} is devoted to the realization theory developed by Y. Yamamoto \cite{YamamotoRealization} and we interpret the control system~\eqref{system_lin_formel2} as an input-output system in this framework. We then prove that the input-output system obtained is \textit{pseudo-rational of order zero} in the sense of Yamamoto. It allows us to characterize the $L^q$ approximate controllability in terms of an approximate left--coprimeness condition on a distributional algebra given in Section~\ref{sec:coprim}. The upper bound on the minimal time of controllability for System~\eqref{system_lin_formel2} is provided in Section~\ref{sec:upper_bound}. Finally, Section~\ref{sec:contro_criterion} summarizes all the sections and it states a necessary and sufficient criterion for the $L^q$ approximate controllability in finite time $T>2d\Lambda_N$ expressed in the complex plane.

\section{Notation}
\label{sec:not}
In this paper, we denote by $\mathbb{N}$ and $\mathbb{N}^*$ the sets of nonnegative and 
positive integers, respectively. We use $\mathbb{R}$, $\mathbb{C}$, $\mathbb{R}_+$, 
and $\mathbb{R}_-$ to denote the sets of real numbers, complex numbers, 
nonnegative, and nonpositive reals respectively. Let $\mathbb{K}$ a commutative ring with unity. Given two positive integers $i$ and $j$, $\mathcal{M}_{i,j}(\mathbb{K})$ is the set of $i \times 
j$ matrices with coefficients in $\mathbb{K}$. For $M \in \mathcal{M}_{i,j}(\mathbb{K})$, we denote by $M^T$ the transposition of the matrix $M$. We use 
$\|\cdot\|$ to denote a norm for every finite-dimensional space (over $\mathbb{K}=\mathbb{C}$) and 
$\vertiii{\,\cdot\,}$ the induced norm for linear maps. The identity matrix in $\mathcal{M}_{i,i}(\mathbb{K})$ is denoted by $I_i$ and the determinant for a square matrix $M \in \mathcal{M}_{i,i}(\mathbb{K})$ is written $\det(M)$. For $M \in \mathcal{M}_{i,j}(\mathbb{K})$, $\rank_{\mathbb{K}} M$ denotes the 
%column 
rank of $M$ over the ring $\mathbb{K}$.
%, \textit{i.e.}, the dimension of the linear space over $\mathbb{K}$ spanned by its columns. 
Given a positive integer $k$, $A \in \mathcal{M}_{i,j}(\mathbb{K})$ and $B\in  \mathcal{M}_{i,k}(\mathbb{K})$, the bracket $\left[A,B\right]$ denotes the juxtaposition of the two matrices, which hence belongs to $\mathcal{M}_{i,j+k}(\mathbb{K})$.

Let $k \in \mathbb{N}^*$ and $q \in [1,+\infty)$. Given an interval $I$ of $\mathbb{R}$, $L^q(I,\mathbb{R}^k)$ represents the space of $q$-integrable functions on the interval $I$ with values in $\mathbb{R}^k$ endowed of the $L^q$-norm on $I$ denoted $\| \cdot \|_{I,\,q}$.
The space of $q$-integrable functions on compact subsets of $\mathbb{R}$ (respectively, $\mathbb{R}_+$) with values in $\mathbb{R}^k$ is denoted $L^q_{\rm loc}\left(\mathbb{R},\mathbb{R}^k\right)$ (respectively,  $L^q_{\rm loc}\left(\mathbb{R}_+,\mathbb{R}^k\right)$). The semi-norms
\[
\|\phi \|_{[0,a],q}:=\left(\int_0^a \vertii{\phi(t)}^q dt  \right)^{1/q},\qquad \phi \in L^q_{\rm loc}\left(\mathbb{R}_+,\mathbb{R}^k\right), \ a\ge 0,
\]
induce a topology on $L^q_{\rm loc}\left(\mathbb{R}_+,\mathbb{R}^k\right)$, which is then a Fr\'echet space. We note by $\pi$ the truncation on positive time for the functions in $L^q_{\rm loc}\left(\mathbb{R},\mathbb{R}^k\right)$, \textit{e.g.}, for $ \phi \in L^q_{\rm loc}\left(\mathbb{R},\mathbb{R}^k\right)$, $\pi(\phi) \in  L^q_{\rm loc}\left(\mathbb{R}_+,\mathbb{R}^k\right)$ and $\pi(\phi) (t)=\phi(t)$ for $t \ge 0$. 
For $d \in \mathbb{N}^*$, we denote by $ L^1(\mathbb{R},\mathbb{R}^{d \times d})$ the space of integrable real $d \times d$ matrix valued maps defined on $\mathbb{R}$ endowed with the norm
\[
\|f \|_{1}:=\int_{-\infty}^{+ \infty} \vertiii{f(t)} dt,\qquad f \in L^1(\mathbb{R},\mathbb{R}^{d \times d}).
\]
The convolution product in $L^1(\mathbb{R},\mathbb{R}^{d \times d})$ is noted $*$ and $f^{\ast k}$ denotes the convolution product of $f \in L^1(\mathbb{R},\mathbb{R}^{d \times d})$ repeated $k \in \mathbb{N}^*$ times. For $I$ an interval of $\mathbb{R}$, we denote by $L^{\infty}(I,\mathbb{R}^{d \times d})$ the space of real $d \times d$ matrix valued maps with a finite essential supremum norm, \textit{i.e.}, $$\mathrm{ess}\,\underset{t \in I}{\sup}\, \vertiii{f(t)}<+\infty.$$

We next introduce the distributional frameworks needed in the paper. A detailed presentation 
with precise definitions can be found, \textit{e.g.}, in 
\cite{schwartz1966theorie,chitour:hal-03827918,fueyo-chitour}. 
 We use $\mathcal{D}(\mathbb{R})$ to denote the space of $C^{\infty}$ functions defined on $\mathbb{R}$ with 
compact support endowed with its usual topology (canonical LF topology). We also use $\mathcal{D}'(\mathbb{R})$ to denote the spaces of continuous linear forms acting on $\mathcal{D}(\mathbb{R})$, \textit{i.e.}, the spaces of all distributions 
on $\mathbb{R}$. For all $\alpha \in \mathcal{D}'(\mathbb{R})$, we note by $\mathrm{supp}\, \mu$ the support of $\alpha$, and $l(\alpha)$ and $r(\alpha)$ represents the infimum and the supremum of the support of $\alpha$ respectively. For $\alpha \in\mathcal{D}'(\mathbb{R})$ and $\psi \in \mathcal{D}(\mathbb{R})$, $\scalprod{\alpha,\psi}_{\mathcal{D}'(\mathbb{R})}$ denotes the duality product. For a sequence of distribution $\left(\alpha_n\right)_{n \in \mathbb{N}} \in \mathcal D'(\mathbb R)$, we say that $\left(\alpha_n\right)_{n \in \mathbb{N}}$ converges toward $\alpha \in \mathcal D'(\mathbb R)$ in a distributional sense if
\begin{equation*}
    \scalprod{\alpha_n, \psi}_{\mathcal{D}'(\mathbb R)} \underset{n \to +\infty}{\longrightarrow}  \scalprod{\alpha, \psi}_{\mathcal{D}'(\mathbb R)},\quad \forall \psi \in \mathcal D(\mathbb R).
\end{equation*}

We note by $\mathcal{E}'(\mathbb{R}_-)$ and $\mathcal{D}'_+(\mathbb{R})$  the subspaces of $\mathcal{D}'(\mathbb{R})$ made of distributions with compact support included in $\mathbb{R}_-$ and with support bounded on the left respectively. Endowed with the convolution product~$*$, $\mathcal{D}'_{+}(\mathbb{R})$ becomes an algebra. 
% Notice that a Radon measure defines also a distribution (of order zero) and, if a Radon measure converges in a distributional sense, then it is also converging in a weak star sense.
% We now extend the definition of the truncation operator to positive times for a distribution in $\mathcal D'_+(\mathbb R)$. Let $\mathcal{D}(\mathbb{R}_+) \subset \mathcal D(\mathbb R)$ be the space of infinitely differentiable functions on $\mathbb{R}$ with compact support contained in $\mathbb{R}_+$ and $\mathcal{D}'(\mathbb{R}_+)$ be its topological dual space. For $\alpha \in \mathcal{D}'_{+}(\mathbb{R})$, we define its truncation $\pi\alpha \in \mathcal D'(\mathbb R_+)$ by
% \begin{equation*}
% \label{eq:pi}
% \scalprod{\pi \alpha, \psi}_{D'(\mathbb R)} = \scalprod{\alpha, \psi}_{D'(\mathbb R)},\qquad \forall \psi \in \mathcal{D}(\mathbb{R}_+).
% \end{equation*}
Given 
 a distribution $\alpha \in\mathcal{ D}'_+(\mathbb{R})$, we use $\widehat{\alpha}(p)$ to denote the two-sided Laplace transform of $\alpha$ at frequency $p\in \mathbb{C}$, provided that % when 
the %previous 
Laplace transform
%against the measure $\mu$ 
exists.
 We denote by $\delta_x \in \mathcal{D}'(\mathbb{R})$ the Dirac distribution at $x \in \mathbb{R}$. All functions $f \in L^{q}_{\rm loc}(\mathbb{R},\mathbb{R})$ define a distribution as follows
 \begin{equation*}
 \begin{aligned}
 \scalprod{f, \psi}_{\mathcal{D}'(\mathbb R)} =\int_{-\infty}^{+\infty} f(\tau) \psi(\tau) d\tau,\quad \forall \psi \in \mathcal D(\mathbb R).   
      \end{aligned}
\end{equation*}
The Laplace transform of the Dirac distribution $\delta_x$, $x \in \mathbb{R}$, and a function $f \in L^q_{\rm loc}(\mathbb{R},\mathbb{R})$ with compact support are equal to $e^{-px}$ and $\int_{-\infty}^{+ \infty} f(t)e^{-pt} dt$, for all $p \in \mathbb{C}$, respectively.

 With a slight abuse of language, we keep the former notations introduced when dealing with matrices whose entries belong to the spaces $L^q_{\rm loc}(\mathbb{R},\mathbb{R})$, $\mathcal D'(\mathbb R)$, $\mathcal E'(\mathbb R_-)$ and $\mathcal D^\prime_+(\mathbb R)$.

\section{Well-posedness, controllability definitions and prerequisites}
\label{sec:pre}
% In this section, we start by defining the functional  spaces considered for the state space and the control of System~\eqref{system_lin_formel2}. Then we introduce the controllability notions that we want to investigate in this paper.

We study the existence and the uniqueness of solutions of System~\eqref{system_lin_formel2}. Otherwise stated, $q$ denotes always an element of $[1,+\infty)$ in the remaining of the paper. Linear difference delay equations with distributed delays fit the framework developed in the book \cite{Hale} with the minor difference that we consider in the present paper solutions in the $L^q$ spaces, $q \in [1,+\infty)$, instead of the space of continuous functions. Therefore we only sketches the proofs of the results stated in this section. We start by considering System~\eqref{system_lin_formel2} without control, \textit{i.e.} the system given by the equation
\begin{equation}
\label{system_lin_formel2_without_input}
 x(t)=\sum_{j=1}^NA_jx(t-\Lambda_j)+\int_0^{\Lambda_N} g(s) x(t-s)ds, \qquad t \ge 0.
\end{equation}

The existence of solutions of System~\eqref{system_lin_formel2_without_input} is summarized in the  following proposition. 

\begin{proposition}
Let $T>0$. For all $\phi \in L^q([-\Lambda_N,0],\mathbb{R}^d)$, there exists a unique solution $x(\cdot)$ belonging to $  L^q([-\Lambda_N,T],\mathbb{R}^d)$ such that $x(\theta)=\phi(\theta)$, for $\theta \in [-\Lambda_N,0]$, and it satisfies \eqref{system_lin_formel2_without_input} for almost all $t \in [0,T]$.

\end{proposition}

\begin{proof}
   The proof is classical and it is done by applying a fixed point theorem. Let $\epsilon>0$ and $\phi \in L^q([-\Lambda_N,0],\mathbb{R}^d)$ , we introduce the operator $T(\phi): L^{q}([-\Lambda_N,\epsilon],\mathbb{R}^d) \longrightarrow L^{q}([-\Lambda_N,\epsilon],\mathbb{R}^d)$ defined by
\begin{equation}
T(\phi)x(t)=\begin{dcases}
 \phi(t), \text{ if $t\in [-\Lambda_N,0)$}, \\
 \sum_{j=1}^NA_jx(t-\Lambda_j)+\int_0^{\Lambda_N} g(s) x(t-s)ds,&\text{ if 
  $t \in [0,\epsilon]$},
\end{dcases}
\end{equation}
with $x(\cdot) \in L^q([-\Lambda_N,\epsilon],\mathbb{R}^d)$. Since $g(\cdot)$ belongs to $L^{\infty}([0,\Lambda_N],\mathbb{R}^d)$, we can take $\epsilon$ enough small so that the operator $T(\phi)$ is strictly contractive. By the Banach fixed point theorem, we get that $T(\phi)$ has a unique fixed point providing a solution of \eqref{system_lin_formel2_without_input} on the interval $[-\Lambda_N,\epsilon]$. One complete the proof of the theorem by successively stepping interval of length $\epsilon+\Lambda_N$.
\end{proof}

For $T \ge 0$, $z(\cdot) \in L^q([-\Lambda_N,T],\mathbb{R}^d)$ and $t \in [0,T]$, we introduce the notation $z_t \in L^q([-\Lambda_N,0],\mathbb{R}^d)$ meaning that $z_t(\theta)=z(t+\theta)$ for $\theta \in [-\Lambda_N,0]$. Let $T\ge 0$, for all $t \in [0,T]$, we define the operator solution of System~\eqref{system_lin_formel2_without_input}  given by
    \begin{align*}
        U_q(t): L^q([-\Lambda_N,0],\mathbb{R}^d) &\longrightarrow L^q([-\Lambda_N,0],\mathbb{R}^d),\\
       \phi \qquad \quad &\mapsto \qquad \quad  U_q(t) \phi=x_t,
    \end{align*}
where $x(\cdot) \in L^q([-\Lambda_N,T],\mathbb{R}^d)$ is the unique function satisfying $x_0=\phi$ and Equation~\eqref{system_lin_formel2_without_input} for almost all $t \in [0,T]$. A fixed point theorem also provide existence and uniqueness of solutions for System~\eqref{system_lin_formel2}, meaning that the control problem is well-posed in $L^q$ spaces.

\begin{proposition}
\label{prop:sol_existence}
    Let $T \ge 0$. For all $\phi \in L^q([-\Lambda_N,0],\mathbb{R}^d)$ and $u \in L^q([0,T],\mathbb{R}^m)$, there exists a unique solution $x(\cdot) \in L^q([-\Lambda_N,T],\mathbb{R}^d)$ such that $x_0=\phi$ and satisfying \eqref{system_lin_formel2} for almost all $t \in [0,T]$.

\end{proposition}

From now on, given  $T > 0$, $u \in L^q([0,T],\mathbb{R}^m)$ and $\phi \in L^q([-\Lambda_N,0],\mathbb{R}^d)$, 
we write 
 $x(\cdot) \in L^q([-\Lambda_N,T],\mathbb{R}^d)$ 
to denote  
% is understood to be 
 the solution given by Proposition~\ref{prop:sol_existence}. We continue this section by defining the controllability notions that we consider in this paper.

\begin{definition}\label{def:Lq-cont} System~\eqref{system_lin_formel2} is:
\begin{enumerate}[1)]
\item \emph{$L^q$ approximately controllable in time $T>0$} if 
%we have the existence of $T$ such that 
for every  $\phi,\psi \in L^q([-\Lambda_N,0],\mathbb{R}^d)$
%, $\phi \in L^2([-\Lambda_N,0],\mathbb{R}^d)$,  
and $\epsilon>0$, there exists $u \in L^q([0,T],\mathbb{R}^m)$ such that 
\[
\|x_T- \psi\|_{[-\Lambda_N,0],q} < \epsilon;
\]
\item \emph{$L^q$ approximately controllable from the origin} if for $\phi\equiv 0$,
% {\color{blue}Pour coller à Yamamoto, je suis obligé de mettre la condition initiale à 0. C'est d'ailleurs cette definition là qui est classique. Si je ne fais pas ça, la Proposition\ref{prop:equi_system} devient fausse. Et je n'ai pas envie d'introduire encore une troisième notion de controllabilité.}
every $\psi \in L^q([-\Lambda_N,0],\mathbb{R}^d)$
%, $\phi \in L^2([-\Lambda_N,0],\mathbb{R}^d)$  
and $\epsilon>0$, there exist $T_{\epsilon,\psi}>0$ and $u \in L^q([0,T_{\epsilon,\psi}],\mathbb{R}^m)$ such that
\[
\|x_{T_{\epsilon,\psi}}- \psi\|_{[-\Lambda_N,0],q} < \epsilon;
\]

\item \emph{$L^q$ approximately controllable from the origin in time $T>0$} if for $\phi\equiv 0$,
% {\color{blue}Pour coller à Yamamoto, je suis obligé de mettre la condition initiale à 0. C'est d'ailleurs cette definition là qui est classique. Si je ne fais pas ça, la Proposition\ref{prop:equi_system} devient fausse. Et je n'ai pas envie d'introduire encore une troisième notion de controllabilité.}
every $\psi \in L^q([-\Lambda_N,0],\mathbb{R}^d)$
%, $\phi \in L^2([-\Lambda_N,0],\mathbb{R}^d)$  
and $\epsilon>0$, there exists $u \in L^q([0,T],\mathbb{R}^m)$ such that
\[
\|x_{T}- \psi\|_{[-\Lambda_N,0],q} < \epsilon.
\]

\end{enumerate}
\end{definition}

One of the major questions concerning the approximate controllability in finite time $T$ is to determine the \textit{minimal time} of controllability.

\begin{definition} We define $T_{ \rm min,q}$ the \textit{minimal time} of the $L^q$ approximate controllability  as follows:
\begin{align*}
T_{ \rm min,q}&:=\inf_{T \in \mathbb{R}_+} \{ \mbox{System~\eqref{system_lin_formel2} is $L^q$ approximately controllable in time $T$}\}.
\end{align*}
\end{definition}

We shall prove in the next pages of the paper that the three notions of controllability given in Definition~\ref{def:Lq-cont} are equivalent and we will give an upper bound on the minimal time of controllability. We show now that the $L^q$ approximate controllability from the origin in time $T>0$ is equivalent to the $L^q$ approximately controllable in time $T>0$. A such result amounts to use a representation formula for solutions of System~\eqref{system_lin_formel2}, called variation-of-constants formula, to state controllability results, see for instance \cite{coron}. In this view, we introduce the \textit{fundamental solution} of System~\eqref{system_lin_formel2_without_input}:

\begin{equation}  
 \label{solution_fondamentale}
X(t)=
\begin{dcases}
 0 \mbox{ for $t<0$}, \\
I_d+\sum\limits_{j=1}^N \seb{A_j} X(t-\Lambda_j) +\int_0^{\Lambda_N} g(s) X(t-s)ds,\quad \mbox{for $t \ge 0$}.
\end{dcases}
 \end{equation}

The map $t \mapsto X(t)$ is a left-continuous function of bounded variation with possible jumps at the set $\{\Lambda_1 n_1+\cdots+ \Lambda_N n_N\,|\, n_1,\cdots,n_N \in \mathbb{N}^N \}$. Thus, for all $T \ge 0$, $t \in [0,T]$ and $u \in L^q([0,T],\mathbb{R}^d)$, the operator $E(t)$ from  $L^q ([0,t],\mathbb{R}^d)$ into $ L^q ([-\Lambda_N,0],\mathbb{R}^d)$ defined by
\begin{equation}
\label{def_Et}
E(t)u(\theta)=-\int_{0^-}^{t^-} d_{\alpha} X(t+\theta-\alpha)Bu(\alpha),\quad \mbox{for $\theta \in [-\Lambda_N,0]$,}
\end{equation}
 has a sense. The integral in \eqref{def_Et} is understood as a Lebesgue-Stieltjes integral on $[0,t)$. It allows to state the following representation formula.

\begin{proposition}
\label{prop:representation}
Let $T\ge 0$. For all $u \in L^q([0,T],\mathbb{R}^m)$ and $\phi \in L^q([-\Lambda_N,0],\mathbb{R}^d)$, the unique solution of System~\eqref{system_lin_formel2}, noted $x(\cdot) \in L^q([-\Lambda_N,T],\mathbb{R}^d)$, satisfies
\begin{equation}
\label{eq:representation}
    x_t=U_q(t)\phi+E(t)u,\quad t\in [0,T].
\end{equation}
\end{proposition}

\begin{proof}
By trading System~\eqref{system_lin_formel2} with a Volterra-Stieltjes equation, the proposition can be deduced from \cite[Chapter 9]{Hale}. 
\end{proof}

In particular, we have the following equivalence between the controllability notions.
% \begin{corollary}
% \label{coro1}
%   System~\eqref{system_lin_formel2} is $L^q$ approximately controllable from the origin in time $T>0$ if, for 
% every $\psi \in L^q([-\Lambda_N,0],\mathbb{R}^d)$
% and $\epsilon>0$, there exists $u \in L^q([0,T],\mathbb{R}^m)$ such that
% \[
% \|E(T)u- \psi\|_{[-\Lambda_N,0],q} < \epsilon.
% \]
% \end{corollary}

\begin{proposition}
\label{coro2}
  System~\eqref{system_lin_formel2} is $L^q$ approximately controllable from the origin in time $T>0$ if and only if it is $L^q$ approximately controllable in time $T>0$.
\end{proposition}

\begin{proof}
First, we notice from Proposition~\ref{prop:representation} that
 System~\eqref{system_lin_formel2} is $L^q$ approximately controllable from the origin in time $T>0$ if, for 
every $\psi \in L^q([-\Lambda_N,0],\mathbb{R}^d)$
and $\epsilon>0$, there exists $u \in L^q([0,T],\mathbb{R}^m)$ such that
\[
\|E(T)u- \psi\|_{[-\Lambda_N,0],q} < \epsilon.
\]

To achieve the proof of the proposition, we have just to show that the $L^q$ approximate controllability from the origin in time $T>0$ implies the $L^q$ approximate controllability in time $T>0$ (the converse is obvious). Let $T>0$ and  $\phi,\psi \in L^q([-\Lambda_N,0],\mathbb{R}^d)$. We want to find a family  of controls $u \in L^q([0,T],\mathbb{R}^m)$ that steers the state $x_T$ as near as we want of $\psi$. Let $\epsilon>0$ and $\tilde{x}_T=U_q(T)\phi$. By the introductory remark of this proof, we can find 
a $u \in L^q([0,T],\mathbb{R}^m)$ such that 
\[
\|E(T)u- \psi+\tilde{x}_T\|_{[-\Lambda_N,0],q} < \epsilon.
\]
Thanks to the representation formula~\eqref{eq:representation} given in Proposition~\ref{prop:representation}, we get that 
\[
\|x_T-\psi\|_{[-\Lambda_N,0],q} < \epsilon,
\]
where $x_T=\tilde{x}_T+E(T)u$. It concludes the proof of the proposition.
\end{proof}

We introduce in the next section the realization theory to undertake the controllability issues of System~\eqref{system_lin_formel2}. This theory will allow to show two properties. On the one hand, we have that the $L^q$ approximate controllability from the origin is equivalent to the $L^q$ approximately controllable from the origin in time $T$, for all $T>2d\Lambda_N$. On the other hand, we are able to prove that the $L^q$ approximately controllable from the origin in time $T$, for all $T>2d\Lambda_N$ is entirely characterized by a frequency criterion. 

\section{Realization theory}

\label{sec:real_theory}

The realization theory has been useful to provide controllability results related to delay systems, see for instance \cite{yamamoto1989reachability,chitour:hal-03827918}. The idea is to interpret the control problem of System~\eqref{system_lin_formel2}, starting from the origin, in terms of an input-output system. The input is then the control of the system applied for negative times and the output is the state (modulo a translation) of System~\eqref{system_lin_formel2} for positive times. More precisely, we consider the following system
\begin{equation}
\label{syst_lin_avec_sortie}
\begin{dcases}
 x(t)=\sum_{j=1}^NA_jx(t-\Lambda_j)+\int_0^{\Lambda_N} g(s)x(t-s)ds+Bu(t),&\text{ for %almost all 
  $t\ge \inf \supp(u)$}, \\
  x(t)=0,&\text{ for %almost all 
  $t<\inf \supp(u)$},\\
y(t)=x(t-\Lambda_N),&\text{ for %almost all 
$t \in [0,+\infty)$},
\end{dcases}
\end{equation}
where the input $u$ belongs to  
\begin{equation*}
\Omega_{q}=\{u\in L^q
%_{\rm loc}
(\mathbb{R},\mathbb{R}^m) \mid 
\text{$\supp(u)\subseteq \mathbb{R}_-$ is compact}
\},
\end{equation*}
%\footnote{\mar{remove loc here and in the two occurrences below?}} 
%made of the functions with 
with $\supp(u)$ denoting the support of $u$.

Among input-output systems, it is easier to study \textit{pseudo-rational} systems. This notion was first introduced by Y. Yamamoto \cite{YamamotoRealization} and a definition for more general systems than System~\eqref{syst_lin_avec_sortie} is given in \cite{YamamotoRealization}. We specify bellow the pseudo-rationality concept to fit the framework of difference delay systems with distributed delays.

\begin{definition}
\label{def_pseudo_rational}
    System~\eqref{syst_lin_avec_sortie} is said to be \textit{pseudo-rational of order zero} if
    there are two \seb{families of} $d\times d$ and $d \times m$ matrices with entries in $\mathcal{E}'(\mathbb{R}_-)$, noted $Q$ and $P$ respectively, such that
     \begin{enumerate}
        
\item \label{Item_def_pseudo_rational1}$Q$ has an inverse over $\mathcal{D}_+'(\mathbb{R})$ in a convolution sense, \textit{e.g.}, there exists a $d \times d$ matrix with entries in $\mathcal{D}_+'(\mathbb{R})$ such that $Q^{-1}*Q=Q*Q^{-1}=\delta_0 I_d$ and the order of the distribution $\det(Q^{-1})$ is zero;
\item \label{Item_def_pseudo_rational2} the output $y(\cdot)$ can be expressed in terms of $Q^{-1}$, $P$ and $u(\cdot)$ as follows
\begin{equation}
\label{representation_radon_measure3}
y(\cdot)=\pi \left(A*u\right)(\cdot),
\end{equation}
where $A:=Q^{-1}*P$ and $\pi \left(A*u\right)$ is the truncation on positive time of the function $A*u$.
        \end{enumerate}
\end{definition}

Let us introduce the two distributions candidate to prove the pseudo-rationality of System~\eqref{syst_lin_avec_sortie}:
\begin{equation}
\label{eq:defQ}
\begin{split}
Q & :=\delta_{-\Lambda_N} I_d- \sum_{j=1}^N \delta_{-\Lambda_N+\Lambda_j} A_j- \delta_{-\Lambda_N}*\tilde{g},\\
P & := B \delta_0,
\end{split}
\end{equation}
where $\tilde{g}$ is the extension of $g$ on $\mathbb{R}$ by zero on the set $(-\infty,0) \cup (\Lambda_N,+\infty)$.

\begin{theorem}
\label{theorem_pseudo_rational}
    System~\eqref{syst_lin_avec_sortie} is pseudo-rational of order zero.
\end{theorem}
\begin{proof}
    We shall first prove that $Q$ has an inverse over $\mathcal{D}_+'(\mathbb{R})$. Assume in a first time that $\vertii{\tilde{g}}_1 \ge 1$. Let $\epsilon \in [0,\Lambda_1)$ such that  we can decompose $\tilde{g}$ as
    \begin{equation}
    \label{eq:th_pseudo_rat1}
   \tilde{g} =\tilde{g}_1+\tilde{g}_2,
     \end{equation}
    where $\tilde{g}_1,\tilde{g}_2 \in L^1(\mathbb{R},\mathbb{R}^{d \times d})$, $\vertii{\tilde{g}_1}_1<1$, $\mathrm{supp}\, \tilde{g}_1 \subseteq [0,\Lambda_1-\epsilon]$ and $\mathrm{supp}\, \tilde{g}_2 \subseteq [\Lambda_1-\epsilon,\Lambda_N]$. Recall that $\phantom{}^{*j}$ is the convolution product repeated $j \in \mathbb{N}^*$ times. Since 
    \begin{equation}
        \vertii{\sum_{j=1}^{+\infty} \tilde{g}_1^{ \ast j}}_1 \le \sum_{j=1}^{+\infty} \vertii{\tilde{g}_1}_1^j<+\infty,
        \end{equation}
  we have that $I_d \delta_0-\tilde{g}_1$ is invertible over $\mathcal{D}_+'(\mathbb{R})$ and its inverse is 
\begin{equation}
    (I_d \delta_0 -\tilde{g}_1)^{-1}=I_d \delta_0+\sum_{j=1}^{+\infty} \tilde{g}_1^{\ast j}.
    \end{equation}
We decompose $Q$ as
\begin{equation}
\label{eq:proof_pseudo}
    Q=I_d\delta_{-\Lambda_N}*(I_d\delta_0-\tilde{g}_1)*(I_d \delta_0+G),
\end{equation}
where $G=(I_d \delta_0-\tilde{g}_1)^{-1}*F_2$ and $F_2= -\sum_{j=1}^N \delta_{\Lambda_j} A_j-\tilde{g}_2$. By the Titchmarsh convolution theorem, for all $k,l \in \{1,\cdots,d\}$, we get that $G_{k,l}\equiv 0$ or $\mathrm{min\, supp}(G_{k,l})>0$ implying that the minimum of the support of the non-zeros elements of the matrices $\left(G^{\ast j}\right)_{j \in \mathbb{N}^*}$ tend to $+\infty$ when $j \to +\infty$. It yields
% \begin{equation}
% \label{eq:pseudo_ration_1}
%     \mathrm{min\, supp}(G)=-\mathrm{min\, supp}(I_d \delta_0-\tilde{g}_1(\cdot))+\mathrm{min \, supp}(F_2).
%     \end{equation}
%     We deduce from Equation~\eqref{eq:pseudo_ration_1} that $$\mathrm{min supp}(I_d \delta_0+G)>0,$$
 that an inverse of $I_d \delta_0+G$ in $\mathcal{D}_+'(\mathbb{R})$ is given by 
 \begin{equation}
 (I_d\delta_0+G)^{-1}=I_d \delta_0+\sum_{j=1}^{+\infty}(-1)^j G^{\ast j}.
 \end{equation}
From Equation~\eqref{eq:proof_pseudo}, we get that $Q^{-1}=I_d \delta_{\Lambda_N}*(I_d \delta_0+G)^{-1}*(I_d \delta_0-\tilde{g}_1)^{-1}$ is an inverse of $Q$ over $\mathcal{D}_+'(\mathbb{R})$ and the order of the distribution $\det(Q^{-1})$ is zero. If $\vertii{\tilde{g}}_1< 1$, we can do the same reasoning by taking $\tilde{g}_1$ equal to $\tilde{g}$ and without considering a function $\tilde{g}_2$ in the decomposition of Equation~\eqref{eq:th_pseudo_rat1}. Thus Item~\ref{Item_def_pseudo_rational1} of Definition~\ref{def_pseudo_rational} is fulfilled.  

It remains to prove that Item~\ref{Item_def_pseudo_rational2} of Definition~\ref{def_pseudo_rational} hold true. Denoting by $\tilde{y}$ the natural extension of the output $y$ on $\mathbb{R}$, \textit{i.e.}, $\tilde{y}(t)=x(t-\Lambda_N)$ for $t \in \mathbb{R}$, Equation~\eqref{syst_lin_avec_sortie} implies that
\begin{equation}
\label{obtention_convolution1}
\left(Q*\tilde{y}\right)(t)=\left(P*u\right)(t),\quad t \in \mathbb{R}.
\end{equation}
We take the convolution product of Equation~\eqref{obtention_convolution1} on the left by $Q^{-1}$ and we obtain
\begin{equation}
\label{obtention_convolution2}
\tilde{y}(t)=\left(Q^{-1}*P*u\right)(t),\quad t \in \mathbb{R}.
\end{equation}
Applying the operator $\pi$ in Equation~\eqref{obtention_convolution2},
%and using the fact that $\tilde{y}(t)=y(t)$ for $t \in \mathbb{R}_+$, 
we have
\begin{equation}
\label{representation_radon_measure4}
y(\cdot)=\pi \left(A*u\right)(\cdot),\qquad \text{where } A:=Q^{-1}*P.
\end{equation}
    It achieves the proof of the theorem. 
\end{proof}

\begin{remark}
   On the one hand this result was stated vaguely in the introduction of the paper \cite{yamamoto1989reachability}. In particular, there was no proof about the existence of the inverse of $Q$. On the other hand, an existence of the inverse of $Q$ over the space of distributions with support bounded on the left is given in \cite{van2003equivalence} in the scalar case but they did not provide the order of $\det(Q^{-1})$. 
   % The first part of the proof of Theorem~\ref{theorem_pseudo_rational} is a specification of the results given in \cite{van2003equivalence} when $Q$ is a matrix with distributions of order zero as entries, \textit{i.e.} Radon measures.
\end{remark}

We define the state space of System~\eqref{syst_lin_avec_sortie} in terms of the distribution $Q$ as
 \begin{equation}
 \label{def_XQ}
 X^{Q,\,q}:=\left\{y \in L^q_{\rm loc}\left(\mathbb{R}_+,\mathbb{R}^d\right)\suchthat \pi(Q*y)=0  \right\}.
 \end{equation}
Thus the system has an input $u$ belonging to $\Omega_q$ and an output $y$ in $ X^{Q,\,q}$. We remark that the set $X^{Q,\,q}$ can be easily identified with the space $L^q\left(\left[0,\Lambda_N\right],\mathbb{R}^d\right)$. In fact, $y \in  X^{Q,\,q}$ if and only if the restriction $y|_{[0,\Lambda_N]}$ is in $L^q([0,\Lambda_N],\mathbb{R}^d)$ and $y$ is the unique extension of $y|_{[0,\Lambda_N]}$ on the interval $[0,+\infty)$ satisfying the condition $\pi(Q*y)=0$.

We now characterize the controllability notions from Definition~\ref{def:Lq-cont} in terms of the above realization theory formalism.
 
%  We introduced the input-output system~\eqref{syst_lin_avec_sortie} to study the controllability of the difference delay system \eqref{system_lin_formel2}. In particular, instead of controlling the state %space 
%  $x(\cdot)$ on $[T-\Lambda_N,T]$, we will control a in state space which is now a subset of $L^q_{\rm loc}\left(\mathbb{R}_+,\mathbb{R}^d\right)$, $ X^{Q,\,q}$.
% Subsection~\ref{subsec:real_theory} introduces formally the realization system associated with System~\eqref{system_lin_formel2} that we consider. Furthermore, we will do the link between the controllability of a such realization system and our System~\eqref{system_lin_formel2}. The remaining part of the section in subsections~\ref{subsec:useful_properties}-\ref{subsec:Approximate_contr_HY}-\ref{subsec:Exact_contr_HY} is devoted to state and prove Hautus--Yamamoto criteria for System~\eqref{system_lin_formel2}, through the realization system that we considered. The section is ended by the proof of our two main results. 
% It remains 
% to relate the controllability properties of System~\eqref{system_lin_formel2} with those of 
% the input-output system~\eqref{syst_lin_avec_sortie}.

\begin{proposition}
\label{def_reachability}
 System~\eqref{system_lin_formel2} is $L^q$ approximately controllable from the origin (respectively, in finite time $T>0$) if and only if for every $\psi \in  X^{Q,\,q}$ there exists a sequence of inputs $(u_n)_{n \in \mathbb{N}} \in (\Omega_q)^{\mathbb{N}}$ (respectively, with support in $[-T,0]$) such that its associated sequence of outputs $(y_n)_{n \in \mathbb{N}} \in \left(L^q_{\rm loc}\left(\mathbb{R}_+,\mathbb{R}^d\right) \right)^{\mathbb{N}}$ through System~\eqref{syst_lin_avec_sortie} satisfies
\[
y_n \underset{n \to +\infty}{\longrightarrow} \psi \quad \text{in} \quad L^q_{\rm loc}\left(\mathbb{R}_+,\mathbb{R}^d\right).
\]

\end{proposition}

\begin{remark}
\label{remark_reachability}
    In the realization theory, the $L^q$ approximate controllability is called the \textit{$X^{Q,q}$ quasi-reachability}.
\end{remark}

We prove in the next section that the $L^q$ approximate controllability from the origin is equivalent to the $L^q$ approximate controllability from the origin in finite time $T>2d\Lambda_N$.

\section{Approximate left-coprimness condition for $L^q$ approximate controllability}

\label{sec:coprim}

We start this section by expounding a formula constructing a family of controls, with the help of particular distributions,  approximating each smooth targets in $X^{Q,q}$ and we provide a characterization of the $L^q$ approximate controllability of System~\eqref{system_lin_formel2} in terms of a left-coprimness condition. These results can be obtained by minor changes of the proof of Theorem~4.4 in \cite{YamamotoRealization}.

\begin{lemma}[Y. Yamamoto]
\label{lemma_propY}
    Assume the existence of two \seb{families of} $d\times d$ and $m \times d$ matrices $(R_n)_{n \in \mathbb{N}}$ and $(S_n)_{n \in \mathbb{N}}$ respectively, with entries in $\mathcal{E}'(\mathbb{R}_-)$, such that
    \begin{equation}
    \label{eq:app_cont0bisbis}
        Q*R_n+P*S_n
        \mathrel{\mathop{\longrightarrow}^{}_{n \rightarrow + \infty}}
        I_d \delta_0
    \end{equation}
    in a distributional sense. Then, for every target output $ \psi \in X^{Q,q}$ with $C^{\infty}$ entries and every $\epsilon>0$, denoting by $\tilde{\psi}$ a $C^{\infty}$ extension of $\psi$ on $\mathbb{R}$ such that $\tilde{\psi} \equiv 0$ on $(-\infty,-\epsilon)$ , we can define the family of inputs $(\omega_n)_{n \in \mathbb{N}} \in \Omega_q$ given by
\begin{equation}\label{eq:motion-planning}
\omega_n = S_n * Q * \tilde{\psi},\quad n \in \mathbb{N},
\end{equation} such that its associated sequence of outputs $(y_n)_{n \in \mathbb{N}} \in \left(L^q_{\rm loc}\left(\mathbb{R}_+,\mathbb{R}^d\right) \right)^{\mathbb{N}}$ through System~\eqref{syst_lin_avec_sortie} satisfies
 \[
 y_n \underset{n \to +\infty}{\longrightarrow} \psi \quad \text{in} \quad L^q_{\rm loc}\left(\mathbb{R}_+,\mathbb{R}^d\right).
 \]
\end{lemma}

\begin{proposition}[Y. Yamamoto]
\label{prop_Y}
   System~\eqref{system_lin_formel2} is $L^q$ approximately controllable from the origin if and only if there exist two \seb{families of} $d\times d$ and $m \times d$ matrices $(R_n)_{n \in \mathbb{N}}$ and $(S_n)_{n \in \mathbb{N}}$ respectively, with entries in $\mathcal{E}'(\mathbb{R}_-)$, such that
    \begin{equation}
    \label{eq:app_cont0}
        Q*R_n+P*S_n
        \mathrel{\mathop{\longrightarrow}^{}_{n \rightarrow + \infty}}
        I_d \delta_0
    \end{equation}
    in a distributional sense.
\end{proposition}

\begin{remark}
   The existence of the two \seb{families of} $d\times d$ and $m \times d$ matrices $(R_n)_{n \in \mathbb{N}}$ and $(S_n)_{n \in \mathbb{N}}$ with entries in $\mathcal{E}'(\mathbb{R}_-)$ satisfying Equation~\eqref{eq:app_cont0} is called the \textit{approximate left--coprimeness} condition in the realization theory formulation. 
\end{remark}

\begin{remark}
   Lemma~\ref{lemma_propY} and Proposition~\ref{prop_Y} are proved in \cite[Theorem 4.1]{YamamotoRealization} for $q=2$ but an easy density argument and the inclusion of the $\left(L^q_{\rm loc}\left(\mathbb{R}_+,\mathbb{R}^d\right)\right)_{q \in [1,+\infty)}$ show that the results hold true as well for $q\in [1,+\infty)$.
\end{remark}

\begin{remark}
\label{remark_con_bis}
    Proposition~\ref{prop_Y} implies that, if System~\eqref{system_lin_formel2} is $L^{q'}$ approximately controllable \seb{from the origin} some $q' \in [1,+\infty)$, then it is $L^{q}$ approximately controllable \seb{from the origin} for all $q \in [1,+\infty)$.
\end{remark}

The theory of Y. Yamamoto does not consider the controllability in finite time. It is our next aim to prove that the $L^q$ approximate controllability from the origin is tantamount to the $L^q$ approximate controllability in finite time and to provide an upper bound on the time of approximate controllability. To conclude this section, we consider the following consequence of Lemma~\ref{lemma_propY} which will be a key element to study the controllability in finite time.

\begin{corollary}
\label{coro_rep}
Let $T>0$. If there exist two family of matrices $(\widetilde{R}_n)_{n \in \mathbb{N}},(\widetilde{S}_n)_{n \in \mathbb{N}}$, with entries in $\mathcal{E}'(\mathbb{R}_-)$, so that the support of the elements of $(\widetilde{S}_n)_{n \in \mathbb{N}}$ belong to $[-T,0]$ and 
 \begin{equation}
    \label{eq:app_cont0bis}
        Q*\widetilde{R}_n+P*\widetilde{S}_n  \mathrel{\mathop{\longrightarrow}^{}_{n \rightarrow + \infty}} \delta_0 I_d\quad \mbox{in $\mathcal{D}'(\mathbb{R})$,}
    \end{equation}
    then System~\eqref{system_lin_formel2} is $L^q$ approximately controllable in time $\tilde{T}$ for all $\tilde{T}$ strictly greater than $T+\Lambda_N$.
\end{corollary}

\begin{proof}
    Under the assumption of the corollary, Lemma~\ref{lemma_propY} implies that, for all $\epsilon>0$, we can approximate each element of $ X^{Q,q}$ with $C^{\infty}$ entries with a family of controls whose supports belong to $[-T-\epsilon-\Lambda_N,0]$. By a density argument, the same property is true for each element of $ X^{Q,q}$ so that System~\eqref{system_lin_formel2} is $L^q$ approximately controllable in time $T+\epsilon+\Lambda_N$, thanks to Proposition~\ref{def_reachability}. It achieves the proof of the corollary.
\end{proof}

We are now ready to prove in the next section that an upper bound for the time minimal of controllability is $2d \Lambda_N$.

\section{Upper bound on the minimal time of controllability}

\label{sec:upper_bound}
 We start this section by citing a controllability lemma due to Kamen \cite[Lemma 6.1]{kamen1976module}. We introduce first some notions of quotient rings. Let $A$ an ideal of $\mathcal{E}'(\mathbb{R}_-)$. We note $\mathcal{E}'(\mathbb{R}_-)/A$ the quotient ring and $[\cdot]$ the equivalence class. 

% For any $\gamma \in \mathcal{D}_+'(\mathbb{R})$ and $\phi \in C^0_c(\mathbb{R})$, the multiplication $\gamma \phi$ is given by $\scalprod{\gamma \phi,\psi}=\scalprod{\gamma ,\phi\psi}$, where $\scalprod{.,.}$ is the duality product on $C^0_c(\mathbb{R})$. We also recall that $l(\cdot)$ denote the infimum of a Radon measure.

\begin{lemma}[Kamen]
\label{lemma_Kamen}
Assume there exists a $\beta \in A$ having an inverse with respect to the convolution in $\mathcal{D}_+'(\mathbb{R})$, \textit{i.e.} there exists $\beta^{-1} \in \mathcal{D}_+'(\mathbb{R})$ such that $\beta * \beta^{-1}=\delta_0$. Then for any $\tau< l(\beta)$ and any $[\omega] \in \mathcal{E}'(\mathbb{R}_-) /A$, there exists an $\alpha \in \mathcal{E}'(\mathbb{R}_-)$ such that $[\alpha]=[\omega]$ and $l(\alpha)>\tau$.
\end{lemma}

If $Q$ and $P$ are scalar ($d=1$), combining Proposition~\ref{prop_Y}, Corollary~\ref{coro_rep} and Lemma~\ref{lemma_Kamen}, we get that the $L^q$ approximate controllability from the origin is equivalent to the $L^q$ approximate controllability from the origin in time $T> 2\Lambda_N$. We next prove that a similar result holds true when $Q$ and $P$ are not necessarily scalar.

\begin{theorem}
\label{theorem_temps_minimal}
    System~\eqref{system_lin_formel2} is $L^q$ approximately controllable from the origin if and only if it is $L^q$ approximately controllable from the origin in time $T$, for all $T>2d\Lambda_N$.
\end{theorem}

\begin{proof}  It is trivially true that the $L^q$ approximate controllability from the origin in time $T$ for all $T>2d \Lambda_N$ implies the $L^q$ approximate controllability from the origin. Let us show the converse assertion. The proof amounts to reduce the approximate left--coprimeness problem to a scalar one and to use Lemma~\ref{lemma_Kamen}. Thanks to Corollary~\ref{coro_rep}, for all $T_c$ such that $T_c>(2d-1) \Lambda_N$, it is sufficient to prove the existence of two family of distribution matrices $(\widetilde{R}_n)_{n \in \mathbb{N}},(\widetilde{S}_n)_{n \in \mathbb{N}}$ with compact supports so that the support of the elements of $(\widetilde{S}_n)_{n \in \mathbb{N}}$ belong to $[-T_c,0]$ and 
 \begin{equation}
    \label{eq:app_cont0bisbi}
        Q*\widetilde{R}_n+P*\widetilde{S}_n  \mathrel{\mathop{\longrightarrow}^{}_{n \rightarrow + \infty}} \delta_0 I_d \quad \mbox{in $\mathcal{D}'(\mathbb{R})$}.
    \end{equation}

Since System~\eqref{system_lin_formel2} is $L^q$ approximately controllable from the origin, applying Proposition~\ref{prop_Y}, there exist 
    two sequences of distribution matrices $(S_n)_{n \in \mathbb{N}}$ and $(R_n)_{n \in \mathbb{N}}$ satisfying \eqref{eq:app_cont0}. Applying a transposition in Equation \eqref{eq:app_cont0}, we get
\begin{equation}
\label{eqmeuhmeuh}
        R_n^T*Q^T+S_n^T* P^T  \mathrel{\mathop{\longrightarrow}^{}_{n \rightarrow + \infty}} \delta_0 I_d \quad \mbox{in $\mathcal{D}'(\mathbb{R})$}.
    \end{equation}
Let us define $F_n=\begin{bmatrix}
R_n^T & S_n^T \end{bmatrix}$ and $G=\begin{bmatrix}
Q & P \end{bmatrix}^T $. Thus, by continuity of the determinant, we deduce from Equation~\eqref{eqmeuhmeuh} that
\begin{equation}
\label{eq:_temp_min1}
      \mathrm{det}\left(F_n*
      G \right)  \mathrel{\mathop{\longrightarrow}^{}_{n \rightarrow + \infty}} \delta_0 \quad \mbox{in $\mathcal{D}'(\mathbb{R})$}.
    \end{equation} 

 Let $\kappa \subset \{1,\cdots,m+d \}$ with cardinal of $\kappa$ equal to $d$. We denote by $ G_{\kappa}$
      the $d \times d$ matrix composed of the $\kappa$ rows of $ G$. In particular, the family $\left(\det(G_{\kappa})\right)_{\kappa,\mathrm{card}\, \kappa =d}$ represents the $d\times d$ minors of the matrix $G$. Applying the Cauchy--Binet formula in Equation~\eqref{eq:_temp_min1}, we get, for all $\kappa$ with $\mathrm{card}\, \kappa=d$, the existence of $\left(\alpha_{\kappa}^n\right)_{\kappa,\mathrm{card}\, \kappa=d} \in \mathcal{E}'(\mathbb{R}_-)$ such that

 \begin{equation}
\label{eq:_temp_min2}
 \sum_{\kappa,\, \mathrm{card}\, \kappa=d} \alpha^n_{\kappa}* \det \left(G_{\kappa} \right) \mathrel{\mathop{\longrightarrow}^{}_{n \rightarrow + \infty}} \delta_0 \quad \mbox{in $\mathcal{D}'(\mathbb{R})$}.
 \end{equation}
 Since $\begin{bmatrix}
           Q &
           P
         \end{bmatrix}$ is pseudo-rational of order zero, we have that $\det \left(G_{\{1,\cdots,d\}}\right)=\det(Q^T)$ is invertible over $\mathcal{D}_+'(\mathbb{R})$ and the support is included in $[-d \Lambda_N,0]$. Let $\widetilde{T}>d \Lambda_N$ to be fixed later and $A=\{\det \left(G_{\{1,\cdots,d\}}\right) * \phi \, |\, \phi \in \mathcal{E}'(\mathbb{R}_-)\}$ the ideal generated by $\det\left(G_{\{1,\cdots,d\}}\right)$ over $\mathcal{E}'(\mathbb{R}_-)$. From Lemma~\ref{lemma_Kamen}, for all $\kappa \neq \{1,...,d\}$, we get the existence of $\mu_{\kappa}^n$ and $\beta_{\kappa}^n$ with compact support in $[-\widetilde{T},0]$ and $\mathbb{R}_-$ respectively such that $\alpha_{\kappa}^n =\mu_{\kappa}^n+\det\left(G_{\{1,\cdots,d\}}\right) *\beta_{\kappa}^n$. Thus, we deduce from \eqref{eq:_temp_min2} that 
\begin{equation}
\label{eq:jesaisplus}
\sum_{\kappa,\, \mathrm{card}\, \kappa=d} \mu^n_{\kappa} * \det \left(G_{\kappa} \right)  \mathrel{\mathop{\longrightarrow}^{}_{n \rightarrow + \infty}} \delta_0 \quad \mbox{in $\mathcal{D}'(\mathbb{R})$},
\end{equation}
         where 
         \begin{equation}
         \mu_{\{1,\cdots,d\}}^n := \alpha_{\{1,\cdots,d\}}^n +\sum_{\substack{\kappa \neq \{1,\cdots,d\}, \\ \mathrm{card}\, \kappa=d}} \beta_{\kappa}^n* \det \left(G_{\kappa} \right).
            \end{equation}
       For $i \in \{1,\cdots,d\}$ and $j \in \{1,\cdots,d+m\}$,  we define the $d \times (d+m)$ matrix $H_n$ whose coefficients are given by
         \begin{equation}
         \label{eq:fin_preuve_supp}
             \left(H_n\right)_{i,j}=\sum_{\substack{\kappa,\, j \in \kappa, \\ \mathrm{card}\, \kappa=d}} \mu_{\kappa}^n *\left(G_{\kappa} \right)^{j,i},
         \end{equation}
where $\left(G_{\kappa} \right)^{j,i}$ denotes the cofactor of the matrix $G_{\kappa}$ associated with the element $G_{j,i}$. Recall that for a square $d \times d$ matrix $A$ with entries in $\mathcal{E}'(\mathbb{R}_-)$, we have $\left(\mathrm{com}\, A\right)^T*A=\det(A) I_d$, where $\mathrm{com}\, A$ is the comatrix of $A$. Then, for all $i,l \in \{1,\cdots,d\}$, we have that 
         \begin{equation}
         \label{eq:jesaisplus2}
             \sum_{1 \le j \le d+m} \left(H_n\right)_{i,j} *G_{j,l}= \sum_{\kappa,\, \mathrm{card}\, \kappa=d} \mu^n_{\kappa}* \left( \sum_{j,\, j\in \kappa} G_{\kappa}^{j,i}*G_{j,l}\right)= e_{i,l}\sum_{\kappa,\, \mathrm{card}\, \kappa=d} \mu_{\kappa}^n *\det(G_{\kappa}),
         \end{equation}
         where $e_{i,l}$ is equal to one if $i=l$ and zero otherwise. Using Equation~\eqref{eq:jesaisplus} in Equation~\eqref{eq:jesaisplus2}, we get that

         \begin{equation}
         \label{eq:jesaisplus2.9}
             H_n *G  \mathrel{\mathop{\longrightarrow}^{}_{n \rightarrow + \infty}} \delta_0 I_d \quad \mbox{in $\mathcal{D}'(\mathbb{R})$}.
         \end{equation}
We can write $H_n=\begin{bmatrix}
             \widetilde{H}_n & \overline{H}_n
         \end{bmatrix},$ where $\widetilde{H}_n$ and $\overline{H}_n$  are $d \times d$ and $d \times m$ matrices respectively. By taking the transposition in Equation~\eqref{eq:jesaisplus2.9}, we obtain
            \begin{equation}
         \label{eq:jesaisplus3}
             Q*\widetilde{H}_n^T+P* \overline{H}_n^T  \mathrel{\mathop{\longrightarrow}^{}_{n \rightarrow + \infty}} \delta_0 I_d \quad \mbox{in $\mathcal{D}'(\mathbb{R})$}.
         \end{equation}
 To achieve the proof, it remains to show that, for all $T_c> (2d-1)\Lambda_N$, the support of $\overline{H}_n^T$ can be chosen to belong to $[-T_c,0]$. Let $T_c> (2d-1)\Lambda_N$, for $i \in \{1,\cdots,d\}$ and $j \in \{d+1,\cdots,d+m\}$, we have that $ j\notin \{1,\cdots,d\}$. Thus, we have that the support of the $\mu_{\kappa}$ and $\left(G_{\kappa} \right)^{j,i}$ appearing in Equation~\eqref{eq:fin_preuve_supp} are included in $[-\widetilde{T},0]$ and $[-(d-1)\Lambda_N,0]$ respectively. It yields that the support of $\left(H_n\right)_{i,j}$ is included in $[-\widetilde{T}-(d\seb{-1)} \Lambda_N,0]$. Since the coefficients of the matrix $\overline{H}_n$ are the $\left(H_n\right)_{i,j}$ for all $i \in \{1,\cdots,d\}$ and $j \in \{d+1,\cdots,d+m\}$, we have that the supports of the elements of $\overline{H}_n^T$ are included in $[-\widetilde{T}-(d-1)\Lambda_N,0]$. Invoking Lemma~\ref{lemma_Kamen}, we can choose $\widetilde{T}$ as near as we want of $d \Lambda_N$ achieving the proof of the theorem.
\end{proof}

\begin{remark}

The proof of Theorem~\ref{theorem_temps_minimal} is inspired by the proof of the corona matrix theorem of P. Fuhrmann \cite{Fuhrmann_corona} (see also the book \cite{nikol2012treatise}). P. Fuhrmann reduced the corona matrix theorem to the one dimensional one of L. Carleson \cite{carleson1962interpolations} by exploiting the determinant as we did to prove Theorem~\ref{theorem_temps_minimal}.
\end{remark}

As an immediate corollary of Proposition~\ref{coro2} and Theorem~\ref{theorem_temps_minimal}, we get the following results on the equivalence between the controllability notions.

\begin{corollary}
    The three controllability notions given in Definition~\ref{def:Lq-cont} are equivalent. Furthermore, the minimal time of approximate controllability is upper bounded by $2d\Lambda_N$.
\end{corollary}

\begin{remark}
It is still an open question to know if the upper bound $2d \Lambda_N$ on the minimal time of controllability is optimal or not. In the current state of the literature, this bound is not optimal when there is no distributed delays in the system ($q \equiv 0$). In fact, in that case, an upper bound is $d\Lambda_N$ as proved in the paper \cite{chitour:hal-03827918}. However, the arguments given in \cite{chitour:hal-03827918} cannot be applied when dealing with distributed delays because they are based on a Cayley--Hamilton theorem for multivariate polynomials. 
\end{remark}

\begin{remark}
    It is worth noting that similar considerations might provide the equivalence between the approximate controllability in finite time $T>0$ and the controllability from the origin for some other delay systems as neutral differential equations for instance.
\end{remark}

Since the upper bound for the minimal time of controllability is obtained, it remains to give a controllability criterion to achieve the goal of this paper.

\section{Frequential $L^q$ approximate controllability criterion in finite time}
\label{sec:contro_criterion}

We start this section by recalling the criterion obtained by Y. Yamamoto for the $L^2$ approximate controllability from the origin for pseudo-rational systems. Since the result is stated in algebraic terms, we introduce the algebra prerequisites. For $\phi \in \mathcal{E}'(\mathbb{R}_-)$, the space of distribution with compact support in $\mathbb{R}_-$, we note $r(\phi)$ the supremum of the support of $\phi$. We denote by $J=\{ \phi \in \mathcal{E}(\mathbb{R}_-);\,r(\phi)<0\}$ the prime ideal consisting of the element of $\mathcal{E}'(\mathbb{R}_-)$ with support strictly negative. Thus we can define the quotient ring $\mathcal{A}= \mathcal{E}(\mathbb{R}_-) /J$ which is an integral domain. It allows to construct the quotient field of $\mathcal{A}$ denoted by $\mathbb{F}$, see before Theorem~3.11 in \cite{yamamoto1989reachability} for more details. Since System~\eqref{system_lin_formel2} is pseudo-rational (of order zero) and taking into account Remark~\ref{remark_reachability}, Theorem~4.1 in the paper \cite{yamamoto1989reachability} reads as follows.

\begin{proposition}[Yamamoto]
\label{th:yamamoto89}
System~\eqref{system_lin_formel2} is $L^2$ approximately controllable from the origin if and only if the two following items hold true:
\begin{enumerate}
    \item \label{item1}$\rank_{\mathbb{C}} [\widehat{Q}(p),\widehat{P}(p)]=d$ for all $p \in \mathbb{C}$;
    \item \label{item2}$\rank_{\mathbb{F}} [Q, P] =d $.
\end{enumerate}
\end{proposition}

In view to give a frequency approximate controllability criterion for System~\eqref{system_lin_formel2}, we need to interpret the algebraic condition given in Item~\ref{item2} of Theorem~\ref{th:yamamoto89} in the frequency domain. We start by proving that the rank of the pair $[Q,P]$ over $\mathbb{F}$ depends on the atomic part at zero of the distribution matrix $Q$ only.

\begin{lemma}
\label{lemma_center}
    The two following items are equivalent:
   \begin{enumerate}
       \item \label{item1}  $\rank_{\mathbb{F}} [Q, P] =d$.
       \item \label{item2}  $\rank_{\mathbb{F}} [A_N \delta_0, P] =d$.
  \end{enumerate}
     
\end{lemma}

\begin{proof}
    % First we notice that for all $W \in \mathcal{E}(\mathbb{R}_-)^{p \times m} $, $\rank_{\mathbb{F}} [W, P]=d$ if and only if there exist a matrix $K$ composed of zeros and ones such that $\rank_{\mathbb{F}} [W+PK]=d$.

    Let us now prove that Item~\ref{item2} implies Item~\ref{item1}. As pointed out by Yamamoto \cite[proof Theorem~4.1]{yamamoto1989reachability}, since $\rank_{\mathbb{F}} [A_N \delta_0, P] =d$, there exist a matrix $K$ composed of zeros and ones such that
    \begin{equation}
    \label{eq:lemma1}
    \rank_{\mathbb{F}} [A_N \delta_0+PK]=d.
    \end{equation}
     Thus, we have that $\det [A_N \delta_0+PK] \neq 0$. For all $\epsilon>0$, there exist two \seb{families of} $d \times d$ matrices, with entries in $\mathcal{E}'(\mathbb{R}_-)$, $B_{\epsilon}$ and $\gamma_{\epsilon}$ respectively, such that 
     \begin{equation}
     \label{eq:lemma2}
   Q=A_N \delta_0+B_{\epsilon}+\gamma_{\epsilon},
   \end{equation}
    where $B_{\epsilon}$ converges to zero in a distributional sense when $\epsilon \to 0$ and $r(\gamma_{\epsilon})<0$. For all $\epsilon$, we have that $Q$ is equal to $A_N \delta_0+B_{\epsilon}$ in the field $\mathbb{F}$. By continuity of the determinant, we deduce that 
    \begin{equation}
   \det \left( Q+PK \right)=\lim\limits_{ \epsilon \to 0}\det \left(A_N\delta_0+B_{\epsilon}+PK\right)=\det \left( A_N\delta_0+PK \right)\neq 0.
     \end{equation}
 It yields that $Q+PK$ is an invertible matrix over $\mathbb{F}$. Thus we have $\rank_{\mathbb{F}} [Q,P]=d$. It achieves to prove that Item~\ref{item2} implies Item~\ref{item1}.

Conversely, we show by contraposition that Item~\ref{item1} implies Item~\ref{item2}. If Item~\ref{item2} is not satisfied then we have a nonzero row vector $\alpha \in \mathcal{M}_{d,1}(\mathbb{F})$ such that 
\begin{equation}
\mbox{$\alpha*A_N\delta_0=0$ and $\alpha*P=0$}.
\end{equation}
For $\epsilon>0$, let us consider the decomposition of $Q=A_N\delta_0+B_{\epsilon}+\gamma_{\epsilon}$ as above. For all $\epsilon>0$, we get that $\alpha*Q=\alpha* B_{\epsilon}$ in $\mathbb{F}$. Since $B_{\epsilon}$ converges toward zero in a distributional sense when $\epsilon \to 0$, we deduce that $\alpha*Q=0$. Thus Item~\ref{item1} is not verified, achieving the proof of the lemma.

%  Conversely, we show that that Item~\ref{item1} implies Item~\ref{item2}. Let us consider the decomposition of $Q=A_N \delta_0+B_{\epsilon}+\gamma_{\epsilon}$ as above.
% From Item~\ref{item1}, we get the existence of $K_{\epsilon}$ a matrix composed of one and zeros such that
%   \begin{equation}
%   \det [A_N \delta_0+B_{\epsilon}+ PK_{\epsilon} ] \neq 0.
%   \end{equation}
% Since the matrices $\left(K_{\epsilon}\right)_{\epsilon>0}$ are composed of zero and ones, we can extract a subsequence converging toward a matrix $K$ with element given by ones and zeros.
%   Thus by continuity of the determinant we get:
%   \begin{equation}
%       0 \neq \det [Q,P]=\lim\limits_{ \epsilon \to + \infty}\det [A_N\delta_0+B_{\epsilon}+PK_{\epsilon}]= \det [A_N \delta_0,P].
%   \end{equation}
%   We deduce that $\rank [A_N \delta_0,0]=d$ achieving the proof of the lemma
\end{proof}

We are now in position to state and prove a frequency approximate controllability criterion for System~\eqref{system_lin_formel2}.

\begin{theorem}
\label{th:main_result_dist}
   For $q \in [1,+\infty)$, System~\eqref{system_lin_formel2} is $L^q$ approximately controllable in time $T>2d\Lambda_N$ if and only if the two following items hold true:
\begin{enumerate}[(a)]
    \item \label{item1bis}$\displaystyle \rank_{\mathbb{C}} \left[e^{p \Lambda_N}\left(I_d-\sum_{j=1}^N A_j e^{-p \Lambda_j} -\int_0^{\Lambda_N} g(s) e^{-ps}ds \right),B\right]=d$ for all $p \in \mathbb{C}$;
    \item \label{item2bis}$\displaystyle \rank_{\mathbb{C}} [A_N, B]=d $.
\end{enumerate}
\end{theorem}

\begin{proof}
First, thanks to Theorem~\ref{theorem_temps_minimal}, Remark~\ref{remark_con_bis} and Proposition~\ref{coro2}, the $L^q$ approximate controllability in time $T>2d\Lambda_N$ is equivalent to the $L^2$ approximate controllability from the origin. For all $p \in \mathbb{C}$, we have that the Laplace transform of $Q$ and $P$ in $p$ are equal to the two following quantities
\begin{equation}
   \widehat{Q}(p)= e^{p \Lambda_N}\left(I_d-\sum_{j=1}^N A_j e^{-p \Lambda_j} -\int_0^{\Lambda_N} g(s) e^{-ps}ds\right) \mbox{  and  } \widehat{P}(p)=B.
\end{equation}
We deduce that Item~\ref{item1} of Proposition~\ref{th:yamamoto89} is equivalent to the condition \ref{item1bis}, while Item~\ref{item2} of Proposition~\ref{th:yamamoto89} is tantamount to Item~\ref{item2bis} by applying Lemma~\ref{lemma_center}. Thus we conclude the proof by using Proposition~\ref{th:yamamoto89}. 
\end{proof}

As a direct illustration of Theorem~\ref{th:main_result_dist}, we present the scalar case. Let \( a_1, a_2 \), and \( b \) be three nonzero real numbers, and consider the scalar delay system with distributed delay and incommensurable delays \( 1 \) and \( \pi \):

\begin{equation}
\label{eq:exemple}
    y(t) = a_1 y(t-1) + a_2 y(t-\pi) + \int_{0}^{\pi} y(t-\tau) \, d\tau + b u(t), \quad t \ge 0.
\end{equation}

It is straightforward to check that System~\eqref{eq:exemple} satisfies the assumptions of Theorem~\ref{th:main_result_dist}. This leads to the following example, which is of particular interest due to the incommensurability of the delays:

\begin{proposition}
    For any \( q \in [1,+\infty) \), System~\eqref{eq:exemple} is \( L^q \)-approximately controllable in time \( T > 2 \pi  \).
\end{proposition}

This result highlights the applicability of our general controllability criterion to systems involving both discrete incommensura delays and distributed delays. We now conclude with a brief summary and perspectives for future work.

\section{Conclusion}

In this paper, we have presented a necessary and sufficient criterion for the $L^q$ approximate controllability in finite time of linear difference delay equations with distributed delays. The results derived offer a comprehensive theoretical framework for analyzing controllability in the frequency domain. However, for systems that are not scalar, further developments are needed. Specifically, it is crucial to design numerical algorithms capable of solving the criterion and determining the controllability of a system in practice. Such algorithms would be essential for analyzing more complex, higher-dimensional systems.

In addition, it would be highly valuable to explore algorithmic approaches that leverage Bézout's identity. This could enable the numerical construction of control inputs that steer the system towards the desired state, thereby enhancing the practical applicability of the theoretical results presented. Future research in this direction could provide significant advancements for the real-world implementation of control strategies in systems governed by delay equations.

%%=============================================%%
%% For submissions to Nature Portfolio Journals %%
%% please use the heading ``Extended Data''.   %%
%%=============================================%%

%%=============================================================%%
%% Sample for another appendix section			       %%
%%=============================================================%%

%% \section{Example of another appendix section}\label{secA2}%
%% Appendices may be used for helpful, supporting or essential material that would otherwise 
%% clutter, break up or be distracting to the text. Appendices can consist of sections, figures, 
%% tables and equations etc.

%%===========================================================================================%%
%% If you are submitting to one of the Nature Portfolio journals, using the eJP submission   %%
%% system, please include the references within the manuscript file itself. You may do this  %%
%% by copying the reference list from your .bbl file, paste it into the main manuscript .tex %%
%% file, and delete the associated \verb+\bibliography+ commands.                            %%
%%===========================================================================================%%

\bibliography{ifacconf}% common bib file
%% if required, the content of .bbl file can be included here once bbl is generated
%%\input sn-article.bbl

\end{document}